\documentclass[12pt]{amsart}

\usepackage[latin1]{inputenc}
\usepackage[english]{babel}
\usepackage{amsthm}
\usepackage{amsmath}
\usepackage{amsxtra}
\usepackage{mathrsfs}
\usepackage{amssymb}
\usepackage[dvips]{graphicx}
\usepackage{graphicx}

\input xy
\xyoption{all}

\newcommand{\ext}{{\rm ext}}
\newcommand{\lev}{{\rm Lev}}
\newcommand{\itp}{super tree property }
\newcommand{\tp}{strong tree property }

\newcommand{\PFA}{{\rm PFA }}
\newcommand{\TP}{{\rm TP }}
\newcommand{\ITP}{{\rm ITP }}
\newcommand{\dom}{{\rm dom}}

\newcommand{\ZFC}{{\rm ZFC }}

\newcommand{\GCH}{{\rm GCH }}

\newcommand{\Add}{{\rm Add}}

\newcommand{\force}{\Vdash}
\newcommand{\spazio}{\textrm{ }}
\newcommand{\restr}{\upharpoonright}

\newcommand{\redux}[1]{\langle #1\rangle}

\newtheorem{theorem}{Theorem}[section]
\newtheorem{lemma}[theorem]{Lemma}
\newtheorem{proposition}[theorem]{Proposition}
\newtheorem{coroll}[theorem]{Corollary}
\newtheorem{remark}[theorem]{Remark}
\newtheorem{definition}[theorem]{Definition}

\newtheorem{claim}[theorem]{Claim}
\begin{document}
\title{Strong Tree Properties For Two Successive Cardinals}


\author[Laura Fontanella ]{Laura Fontanella}
\email{fontanella@logique.jussieu.fr}
\address{Equipe de Logique Math\'ematique,
Universit\'e Paris Diderot Paris 7, UFR de math\'ematiques case 7012, 
site Chevaleret, 75205 Paris Cedex 13, France}



\maketitle

\begin{abstract}{}
An inaccessible cardinal $\kappa$ is supercompact when $(\kappa, \lambda)$-ITP holds for all $\lambda\geq \kappa.$ We prove that if there is a model of $\ZFC$ with two supercompact cardinals, then there is a model of \ZFC where simultaneously $(\aleph_2, \mu)$-ITP and $(\aleph_3, \mu')$-ITP hold, for all $\mu\geq \aleph_2$ and $\mu'\geq \aleph_3.$
   \end{abstract}

\


{\bf Key Words:} tree property, large cardinals, forcing.

{\bf Mathematical Subject Classification:} 03E55.\\

\section{Introduction}

The result presented in this paper concern two combinatorial properties that generalize the usual tree property for a regular cardinal. It is a well known fact that an inaccessible cardinal is weakly compact if, and only if, it satisfies the tree property. A similar characterization was made by Jech \cite{Jech} and Magidor \cite{Magidor} for strongly compact and supercompact cardinals; we will refer to the corresponding combinatorial properties as the \emph{strong tree property} and the \emph{super tree property}. Thus, an inaccessible cardinal is strongly compact if, and only if, it satisfies the strong tree property (see Jech \cite{Jech}), while it is supercompact if, and only if, it satisfies the super tree property (see Magidor \cite{Magidor}).\\    

While the previous results date to the early $1970$s, it was only recently that a systematic study of these properties was undertaken by Weiss (see \cite{WeissPhd} and \cite{Weiss}). Although the strong tree property and the super tree property characterize large cardinals, they can be satisfied by small cardinals as well. Indeed, Weiss proved in \cite{Weiss} that for every $n\geq 2,$ one can define a model of the super tree property for $\aleph_n,$ starting from a model with a supercompact cardinal. Since the super tree property captures the combinatorial essence of supercompact cardinals, then we can say that in Weiss model, $\aleph_n$ is in some sense supercompact.\\ 


By working on the super tree property at $\aleph_2,$ Viale and Weiss (see \cite{VialeWeiss} and \cite{Viale}) obtained new results about the consistency strength of the Proper Forcing Axiom. They proved that if one forces a model of \PFA using a forcing that collapse $\kappa$ to $\omega_2$ and satisfies the $\kappa$-covering and the $\kappa$-approximation properties, then $\kappa$ has to be strongly compact; if the forcing is also proper, then $\kappa$ is supercompact. Since every known forcing producing a model of \PFA by collapsing $\kappa$ to $\omega_2$ satisfies those conditions, we can say that the consistency strength of \PFA is, reasonably, a supercompact cardinal.\\


 

It is natural to ask whether two small cardinals can simultaneously have the strong or the super tree properties. Abraham define in \cite{Abraham} a forcing construction producing a model of the tree property for $\aleph_2$ and $\aleph_3,$  starting from a model of $\ZFC+\GCH$ with a supercompact cardinal and a weakly compact cardinal above it. Cummings and Foreman \cite{CummingsForeman} proved that if there is a model of set theory with infinitely many supercompact cardinals, then one can obtain a model in which every $\aleph_n$ with $n\geq 2$ satisfies the tree property.\\ 

In the present paper, we construct a model of set theory in which $\aleph_2$ and 
$\aleph_3$ simultaneously satisfy the super tree property, starting from a model of $\ZFC$ with two supercompact cardinals $\kappa<\lambda.$ We will collapse $\kappa$ to $\aleph_2$ and $\lambda$ to $\aleph_3,$ in such a way that they will still satisfy the super tree property. The definition of the forcing construction required for that theorem is motivated by Abraham \cite{Abraham} and Cummings-Foreman \cite{CummingsForeman}. We also conjecture that in the model defined by Cummings and Foreman, every $\aleph_n$ (with $n\geq 2$) satisfies the super tree property.\\

The paper is organized as follows. In \S \ref{sstp} we introduce the strong and the super tree properties. In \S \ref{mainforcing}, \S \ref{bricks} and \S \ref{bricks2} we define the forcing notion required for the final theorem and we discuss some properties of that forcing. \S \ref{BPP} is devoted to the proof of two preservation theorems. Finally, the proof of the main theorem is developed in \S \ref{theorem}.

\section{Preliminaries and Notation}

Given a forcing $\mathbb{P}$ and conditions $p,q\in \mathbb{P},$ we use $p\leq q$ in the sense that $p$ is stronger than $q.$ A poset $\mathbb{P}$ is \emph{separative} if whenever 
$q\not\leq p,$ then some extension of $q$ in $\mathbb{P}$ is incompatible with $p.$ Every partial order can be turned into a separative poset. Indeed, one can define $p\prec q$ iff all extensions of $p$ are compatible with $q,$ then the resulting equivalence relation, given by $p\sim q$ iff $p\prec q$ and $q\prec p,$ provides a separative poset; we denote by $[p]$ the equivalence class of $p.$\\ 

Given two forcings $\mathbb{P}$ and $\mathbb{Q},$ we will write 
$\mathbb{P}\equiv \mathbb{Q}$ when $\mathbb{P}$ and $\mathbb{Q}$ are equivalent, namely: 
\begin{enumerate}
\item for every filter $G_{\mathbb{P}}\subseteq \mathbb{P}$ which is $V$-generic over $\mathbb{P},$ there exists a filter $G_{\mathbb{Q}}\subseteq \mathbb{Q}$ which is $V$-generic over $\mathbb{Q}$ and $V[G_{\mathbb{P}}]= V[G_{\mathbb{Q}}] ;$
\item for every filter $G_{\mathbb{Q}}\subseteq \mathbb{Q}$ which is $V$-generic over $\mathbb{Q},$ there exists a filter $G_{\mathbb{P}}\subseteq \mathbb{P}$ which is $V$-generic over $\mathbb{P}$ and $V[G_{\mathbb{P}}]= V[G_{\mathbb{Q}}].$
\end{enumerate}

If $\mathbb{P}$ is any forcing and $\dot{\mathbb{Q}}$ is a $\mathbb{P}$-name for a forcing, then we denote by $\mathbb{P}\ast \dot{\mathbb{Q}}$ the poset 
$\{(p,q);\spazio p\in \mathbb{P}, q\in V^{\mathbb{P}}\textrm{ and }p\force q\in \dot{\mathbb{Q}}\},$
where for every $(p,q), (p',q')\in \mathbb{P}\ast \dot{\mathbb{Q}},$ $(p,q)\leq (p',q')$ if, and only if, $p\leq p'$ and $p\force q\leq q'.$\\

 If $\mathbb{P}$ and $\mathbb{Q}$ are two posets, a \emph{projection} $\pi: \mathbb{Q}\to \mathbb{P}$ is a function such that:
\begin{enumerate}
\item for all $q,q'\in \mathbb{Q},$ if $q\leq q',$ then $\pi(q)\leq \pi(q');$
\item $\pi(1_\mathbb{Q})= 1_{\mathbb{P}};$
\item for all $q\in \mathbb{Q},$ if $p\leq \pi(q),$ then there is $q'\leq q$ such that $\pi(q')\leq p.$ 
\end{enumerate}

If $\pi: \mathbb{Q}\to \mathbb{P}$ is a projection and $G_\mathbb{P}\subseteq \mathbb{P}$ is a $V$-generic filter, define
$$\mathbb{Q}/G_{\mathbb{P}}:=\{ q\in \mathbb{Q};\ \pi(q)\in G_\mathbb{P}\},$$
$\mathbb{Q}/G_\mathbb{P}$ is ordered as a subposet of $\mathbb{Q}.$ The following hold: 
\begin{enumerate}
\item If $G_{\mathbb{Q}}\subseteq \mathbb{Q}$ is a generic filter over  $V$ and $H:= \{p\in \mathbb{P};\ \exists q\in G_{\mathbb{Q}}(\pi(q)\leq p)  \},$ then $H$ is $\mathbb{P}$-generic over $V;$
\item if $\mathbb{G}_{\mathbb{P}}\subseteq \mathbb{P}$ is a generic filter over $V,$ and if $G\subseteq \mathbb{Q}/G_\mathbb{P}$ is a generic filter over $V[\mathbb{G_{\mathbb{P}}}],$ then $G$ is $\mathbb{Q}$-generic over $V,$ and $\pi''[G]$ generates $G_{\mathbb{P}};$
\item if $G_{\mathbb{Q}}\subseteq \mathbb{Q}$ is a generic filter, and $H:= \{p\in \mathbb{P};\ \exists q\in G_{\mathbb{Q}}(\pi(q)\leq p) \},$ then $\mathbb{G}_{\mathbb{Q}}$ is 
$\mathbb{Q}/G_\mathbb{P}$-generic over $V[H].$ That is, we can factor forcing with $\mathbb{Q}$ as forcing with $\mathbb{P}$ followed by forcing with 
$\mathbb{Q}/G_\mathbb{P}$ over $V[\mathbb{G}_{\mathbb{P}}].$
\end{enumerate}

Some of our projections $\pi: \mathbb{Q}\to \mathbb{P}$ will also have the following property: for all $p\leq \pi(q),$ there is $q'\leq q$ such that 
\begin{enumerate}
\item $\pi(q')= p,$  
\item for every $q^*\leq q,$ if $\pi(q^*)\leq p,$ then $q^*\leq q'.$
\end{enumerate}
We denote by $\ext(q,p)$ any condition like $q'$ above (if a condition $q''$ satisfies the previous properties, then $q'\leq q''\leq q'$). In this case, if $G_\mathbb{P}\subseteq \mathbb{P}$ is a generic filter, we can define an ordering on $\mathbb{Q}/ G_{\mathbb{P}}$ as follows: $p\leq^* q$ if, and only if, there is $r\leq \pi(p)$ such that $r\in G_{\mathbb{P}}$ and $\ext(p,r)\leq q.$ Then, forcing over 
$V[G_{\mathbb{P}}]$ with $\mathbb{Q}/G_{\mathbb{P}}$ ordered as a subposet of $\mathbb{Q},$ is equivalent to forcing over $V[G_{\mathbb{P}}]$ with $(\mathbb{Q}/ G_{\mathbb{P}}, \leq^*).$ 

\

Let $\kappa$ be a regular cardinal and $\lambda$ an ordinal, we denote by $\Add(\kappa, \lambda)$ the poset of all partial functions $f:\lambda\to 2$ of size less than $\kappa,$ ordered by reverse inclusion. We use $\Add(\kappa)$ to denote $\Add(\kappa, \kappa).$\\ 


If $V\subseteq W$ are two models of set theory with the same ordinals and $\eta$ is a cardinal in $W,$ we say that $(V,W)$ has the $\eta$-covering property if, and only if, every set $X\subseteq V$ in $W$ of cardinality less than $\eta$ in $W,$ is contained in a set $Y\in V$ of cardinality less than $\eta$ in $V.$\\ 

Assume that $\mathbb{P}$ is a forcing notion, we will use $\redux{\mathbb{P}}$ to denote the canonical $\mathbb{P}$-name for a $\mathbb{P}$-generic filter over $V.$\\

\begin{lemma} (Easton's Lemma) Let $\kappa$ be regular. If $\mathbb{P}$ has the $\kappa$-chain condition and $\mathbb{Q}$ is $\kappa$-closed, then 
\begin{enumerate}
\item $\force_{\mathbb{Q}} \mathbb{P}\textrm{ has the $\kappa$-chain condition};$ 
\item $\force_{\mathbb{P}} \mathbb{Q}\textrm{ is a $<\kappa$-distributive};$
\item If $G$ is $\mathbb{P}$-generic over $V$ and $H$ is $\mathbb{Q}$-generic over $V,$ then $G$ and $H$ are mutually generic;
\item If $G$ is $\mathbb{P}$-generic over $V$ and $H$ is $\mathbb{Q}$-generic over $V,$ then $(V, V[G][H])$ has the $\kappa$-covering property;
\item If $\mathbb{R}$ is $\kappa$-closed, then $\force_{\mathbb{P}\times \mathbb{Q}} \textrm{ $\mathbb{R}$ is $<\kappa$-distributive}.$
\end{enumerate}  
\end{lemma}

For a proof of that lemma see \cite[Lemma 2.11]{CummingsForeman}.\\  

Let $\eta$ be a regular cardinal, $\theta> \eta$ be large enough and $M\prec H_\theta$ of size $\eta$. We say that $M$ is \emph{internally approachable of length $\eta$} if it can be written as the union of an increasing continuous chain  $\langle M_{\xi}: \xi < \eta\rangle $ of elementary submodels of $H(\theta)$ of size less than $\eta,$ such that $\langle M_{\xi}: \xi < \eta' \rangle  \in M_{\eta' +1}$, for every ordinal $\eta' < \eta.$\\

\begin{lemma}($\Delta$-system Lemma) Assume that $\lambda$ is a regular cardinal and $\kappa<\lambda$ is such that $\alpha^{<\kappa}<\lambda,$ for every 
$\alpha<\lambda.$ Let $\mathscr{F}$ be a family of sets of cardinality less than $\kappa$ such that 
$\vert \mathscr{F}\vert= \lambda.$ There exists a family $\mathscr{F}'\subseteq \mathscr{F}$ of size $\lambda$ and a set $R$ such that $X\cap Y= R,$ for any two distinct 
$X,Y\in \mathscr{F}'.$ \end{lemma}

For a proof of that lemma see \cite{Kunen}.

\begin{lemma}(Pressing Down Lemma) If $f$ is a regressive function on a stationary set $S\subseteq [A]^{<\kappa}$ (i.e. $f(x)\in x,$ for every non empty $x\in S$), then there exists a stationary set $T\subseteq S$ such that $f$ is constant on $T.$
\end{lemma}

For a proof of that lemma see \cite{Kunen}.\\ 


We will assume familiarity with the theory of large cardinals and elementary embeddings, as developed for example in \cite{Kanamori}. 

\begin{lemma} (Laver) \cite{Laver} If $\kappa$ is a supercompact cardinal, then there exists $L: \kappa \to V_{\kappa}$ such that: for all $\lambda,$ for all $x\in H_{\lambda^+},$ there is $j: V\to M$
such that $j(\kappa)>\lambda,$ ${}^\lambda M\subseteq M$ and $j(L)(\kappa)= x.$
\end{lemma}

\begin{lemma} (Silver) Let $j: M\to N$ be an elementary embedding between inner models of {\rm ZFC}. Let $\mathbb{P}\in M$ be a forcing and suppose that $G$ is $\mathbb{P}$-generic over $M,$ $H$ is $j(\mathbb{P})$-generic over $N,$ and $j''[G]\subseteq H.$ Then, there is a unique $j^*: M[G]\to N[H]$ such that $j^*\restr M= j$ and $j^*(G)= H.$ 
\end{lemma}

\begin{proof} If $j''[G]\subseteq H,$ then the map $j^*(\dot{x}^{G})= j(\dot{x})^{H}$ is well defined and satisfies the required properties. \end{proof}

\section{The Strong and the Super Tree Properties}\label{sstp}

In order to define the \tp and the \itp for a regular cardinal $\kappa\geq \aleph_2,$ we need to define the notion of \emph{$(\kappa, \lambda)$-tree}, for an ordinal $\lambda\geq \kappa.$

\begin{definition}\label{main definition} Given $\kappa\geq \omega_2$ a regular cardinal and $\lambda\geq \kappa,$ a \emph{$(\kappa, \lambda)$-tree} is a set $F$ satisfying the following properties:  
\begin{enumerate}
\item for every $f\in F,$ $f: X\to 2,$ for some $X\in [\lambda]^{<\kappa}$
\item for all $f\in F,$ if $X\subseteq \dom(f),$ then $f\restr X\in F;$
\item the set $\lev_X(F):= \{f\in F;\spazio \dom(f)=X \}$ is non empty, for all $X\in [\lambda]^{<\kappa};$
\item $\vert \lev_X(F) \vert<\kappa ,$ for all $X\in [\lambda]^{<\kappa}.$
\end{enumerate}
\end{definition}

When there is no ambiguity, we will simply write $\lev_X$ instead of $\lev_X(F).$ 

\begin{definition}\label{branches} Given $\kappa\geq \omega_2$ a regular cardinal, $\lambda\geq \kappa,$ and $F$ a $(\kappa, \lambda)$-tree,
\begin{enumerate}
\item  a \emph{cofinal branch} for $F$ is a function $b: \lambda \to 2$ such that $b\restr X\in \lev_X(F),$ for all $X\in[\lambda]^{<\kappa};$
\item an \emph{$F$-level sequence} is a function $D: [\lambda]^{<\kappa}\to F$ such that for every $X\in [\lambda]^{<\kappa},$ $D(X)\in \lev_X(F);$
\item given an $F$-level sequence $D,$ an \emph{ineffable branch} for $D$ is a cofinal branch $b: \lambda \to 2$ such that
$\{X\in [\lambda]^{<\kappa};\spazio b\restr X= D(X) \}$ is stationary. 
\end{enumerate}
\end{definition}

\begin{definition} Given $\kappa\geq \omega_2$ a regular cardinal and $\lambda\geq \kappa,$ 
\begin{enumerate}
\item $(\kappa, \lambda)$-\TP holds if every $(\kappa, \lambda)$-tree has a cofinal branch;
\item $(\kappa, \lambda)$-\ITP holds if for every $(\kappa, \lambda)$-tree $F$ and for every $F$-level sequence $D,$ there is an an ineffable branch for $D;$
\item we say that $\kappa$ satisfies the \emph{strong tree property} if $(\kappa, \mu)$-\TP holds, for all $\mu\geq \kappa;$
\item we say that $\kappa$ satisfies the \emph{super tree property} if 
$(\kappa,\mu)$-\ITP holds, for all $\mu\geq \kappa;$
\end{enumerate}
\end{definition}


\section{The Preservation Theorems}\label{BPP}

It will be important in what follows that certain forcings cannot add ineffable branches. 

\begin{theorem}\label{closed}(First Preservation Theorem) Let $\theta$ be a regular cardinal and $\mu\geq \theta$ be any ordinal. Assume that $F$ is a $(\theta, \mu)$-tree and $\mathbb{Q}$ is an $\eta^+$-closed forcing with $\eta< \theta\leq 2^{\eta}.$ For every filter $G_{\mathbb{Q}}\subseteq \mathbb{Q}$ generic over $V,$ every cofinal branch for $F$ in $V[G_{\mathbb{Q}}]$ is already in $V.$
\end{theorem}

\begin{proof} We can assume, without loss of generality, that $\eta$ is minimal such that $2^{\eta}\geq \theta.$ Assume towards a contradiction that $\mathbb{Q}$ adds a cofinal branch to $F,$ let $\dot{b}$ be a $\mathbb{Q}$-name for such a function.  
For all $\alpha\leq \eta$ and all $s\in {}^{\alpha}2,$ we are going to define by induction three objects $a_{\alpha}\in [\mu]^{<\theta},$ $f_s\in \lev_{a_{\alpha}}$ and $p_s\in \mathbb{Q}$ such that:

\begin{enumerate}
\item $p_s\force \dot{b}\restr a_{\alpha}= f_s;$ 
\item $f_{s\smallfrown 0}(\beta)\neq f_{s\smallfrown 1}(\beta),$ for some $\beta<\mu;$
\item if $s\subseteq t,$ then $p_t\leq p_s;$
\item if $\alpha< \beta,$ then $a_{\alpha}\subset a_{\beta}.$
\end{enumerate}

Let $\alpha<\eta,$ assume that $a_{\alpha}, f_s$ and $p_s$ have been defined for all $s\in {}^{\alpha}2.$ We define $a_{\alpha+1},$ $f_s,$ and $p_s,$ for all $s\in {}^{\alpha+1}2.$ Let $t$ be in ${}^{\alpha}2,$ we can find an ordinal $\beta_t\in \mu$ and two conditions $p_{t\smallfrown 0}, p_{t\smallfrown 1}\leq p_t$ such that $p_{t\smallfrown 0}\force \dot{b}(\beta_t)= 0$ and $p_{t\smallfrown 1}\force \dot{b}(\beta_t)= 1.$ (otherwise, $\dot{b}$ would be a name for a cofinal branch which is already in $V$).  Let 
$a_{\alpha+1}:=a_{\alpha}\cup \{\beta_t;\ t\in {}^{\alpha}2\},$ then $\vert a_{\alpha+1}\vert <\theta,$ because $2^{\alpha}<\theta$. We just defined, for every $s\in {}^{\alpha+1}2,$ a condition $p_s.$ Now, by strengthening $p_s$ if necessary, we can find $f_s\in \lev_{a_{\alpha+1}}$ such that 
$$p_s\force \dot{b}\restr a_{\alpha+1}= f_s.$$
Finally, $f_{t\smallfrown 0}(\beta_t)\neq f_{t\smallfrown 1}(\beta_t),$ for all $t\in {}^{\alpha}2:$ because $p_{t\smallfrown 0}\force f_{t\smallfrown 0}(\beta_t)= \dot{b}(\beta_t)=0,$ while $p_{t\smallfrown 1}\force f_{t\smallfrown 1}(\beta_t)= \dot{b}(\beta_t)=1.$\\

If $\alpha$ is a limit ordinal $\leq \eta,$ let $t$ be any function in ${}^{\alpha}2.$ Since $\mathbb{Q}$ is $\eta^+$-closed, there is a condition $p_t$ such that $p_t\leq p_{t\restr {\beta}},$ for all $\beta<\alpha.$ Define $a_{\alpha}:= \underset{\beta<\alpha}{\bigcup} a_{\beta}.$ By strengthening $p_t$ if necessary, we can find $f_t\in \lev_{a_{\alpha}}$ such that 
$p_t\force \dot{b}\restr a_{\alpha}=f_t.$ That completes the construction. \\

We show that $\vert \lev_{a_{\eta}}\vert\geq {}^{\eta}2\geq \theta,$ thus a contradiction is obtained. Let $s\neq t$ be two functions in ${}^{\eta}2,$ we are going to prove that 
$f_s\neq f_t.$ Let $\alpha$ be the minimum ordinal less than $\eta$ such that $s(\alpha)\neq t(\alpha),$ without loss of generality $r\smallfrown 0\sqsubset s$ and $r\smallfrown 1\sqsubset t,$ for some $r\in {}^{\alpha}2.$ By construction, 
$$p_s\leq p_{r\smallfrown 0}\force \dot{b}\restr a_{\alpha+1}= f_{r\smallfrown 0}\textrm{ and }p_t\leq p_{r\smallfrown 1}\force \dot{b}\restr a_{\alpha+1}= f_{r\smallfrown 1},$$ where 
$f_{r\smallfrown 0}(\beta)\neq f_{r\smallfrown 1}(\beta),$ for some $\beta.$ Moreover, $p_s\force \dot{b}\restr a_{\eta}=f_s$ and $p_t\force \dot{b}\restr a_{\eta}= f_t,$ hence 
$f_s\restr a_{\alpha+1}(\beta)= f_{r\smallfrown 0}(\beta)\neq f_{r\smallfrown 1}(\beta)= f_t\restr a_{\alpha+1}(\beta),$ thus $f_s\neq f_t.$ That completes the proof. \end{proof}

\begin{coroll}\label{primo} Let $\theta$ be a regular cardinal and $\mu\geq \theta$ be any ordinal. Assume that $F$ is a $(\theta, \mu)$-tree and $D$ is an $F$-level sequence, and suppose that 
$\mathbb{Q}$ is an $\eta^+$-closed forcing with $\eta<\theta\leq 2^{\eta}.$ For every filter $G_{\mathbb{Q}}\subseteq \mathbb{Q}$ generic over $V,$ if $D$ has no ineffable branches in $V,$ then there are no ineffable branches for $D$ in $V[G_{\mathbb{Q}}].$
\end{coroll}

\begin{proof} Assume that $b\in V[G_\mathbb{Q}]$ is an ineffable branch for $D.$ By Proposition \ref{closed}, we have $b\in V.$ Define in $V$ the set 
$$S:=\{X\in [\mu]^{<\theta};\spazio b\restr X=D(X) \}.$$
Then, $S$ is stationary in $V[G_\mathbb{Q}],$ hence it is stationary in $V.$ Thus $b$ is an ineffable branch for $D$ in $V.$ \end{proof}

The following proposition is rather ad hoc. It will be used several times in the final theorem.  

\begin{theorem}\label{mafalda}(Second Preservation Theorem) Let $V\subseteq W$ be two models of set theory with the same ordinals and let $\mathbb{P}\in V$ be a forcing notion and $\kappa$ a cardinal in $V$ such that: 
\begin{enumerate}
\item $\mathbb{P}\subseteq \Add(\aleph_n, \tau)^V,$ for some $\tau>\aleph_n,$\\
 and for every $p\in \mathbb{P},$ if $X\subseteq \dom(p),$ then $p\restr X\in \mathbb{P},$
\item $\aleph_{m}^V=\aleph_m^W,$ for every $m\leq n,$ and $W\models \kappa=\aleph_{n+1},$
\item $(V,W)$ has the $\kappa$-covering property,
\item in $V,$ we have $\gamma^{<\aleph_n}<\kappa,$ for every cardinal $\gamma<\kappa.$ 
\end{enumerate}
Let $F\in W$ be a $(\aleph_{n+1},\mu )$-tree with $\mu\geq \aleph_{n+1},$ then for every filter $G_{\mathbb{P}}\subseteq \mathbb{P}$ generic over $W,$ every cofinal branch for $F$ in $W[G_{\mathbb{P}}]$ is already in $W.$
\end{theorem}

\begin{proof} Work in $W.$ Let $\dot{b}\in W^{\mathbb{P}}$ and let $p\in \mathbb{P}$ such that 
$$p\force \dot{b}\textrm{ is a cofinal branch for $F.$}$$ 

We are going to find a condition $q\in \mathbb{P}$ such that $q\vert \vert p$ and for some $b\in W,$ we have $q\force \dot{b}= b.$
Let $\chi$ be large enough, for all $X\prec H_{\chi}$ of size $\aleph_n,$ we fix a condition $p_X\leq p$ and a function $f_X\in Lev_{X\cap\mu}$ such that 
$$p_X\force \dot{b}\restr X= f_X.$$
 Let $S$ be the set of all the structures $X\prec H_{\chi},$ such that $X$ is internally approachable of length $\aleph_n.$ 
Since every condition of $\mathbb{P}$ has size less than $\aleph_n,$ then for all $X\in S,$ there is $M_X\in X$ of size less than $\aleph_n$ such that 
$$p_X\restr X\subseteq M_X.$$
  By the Pressing Down Lemma, there exists $M^*$ and a stationary set $E^*\subseteq S$ such that $M^*=M_X,$ for all $X\in E^*.$ The set $M^*$ has size less than $\kappa$ in $W,$ hence $A:= \underset{X\in E^*}{\bigcup} p_X\restr M^*$ has size $<\kappa$ in $W.$ By the assumption, $A$ is covered by some $N\in V$ of size 
$\gamma<\kappa$ in $V.$ In $V,$ we have 
$\vert [N]^{<\aleph_n}\vert=  \gamma^{<\aleph_n}< \kappa.$ It follows that in $W$ there are less than $\kappa$ (hence less than $\aleph_{n+1}$) possible values for 
$p_X\restr M^*.$ Therefore, we can find in $W$ a cofinal $E\subseteq E^*$ and a condition $q\in \mathbb{P},$ such that 
$p_X\restr X=q,$ for all $X\in E.$

\begin{claim} $f_X\restr Y= f_Y\restr X,$ for all $X,Y\in E.$
\end{claim}

\begin{proof} Let $X,Y\in E,$ there is $Z\in E$ with $X,Y,\dom(p_X), \dom(p_Y)\subseteq Z.$ Then, we have $p_X\cap p_Z= p_X\cap p_Z\restr Z= p_X\cap q= q,$ thus $p_X\vert \vert p_Z$ and similarly $p_Y\vert\vert p_Z.$ Let $r\leq p_X, p_Z$ and $s\leq p_Y, p_Z,$ then 
$r\force f_Z\restr X= \dot{b}\restr X= f_X$ and $s\force f_Z\restr Y= \dot{b}\restr Y= f_Y.$ It follows that $f_X\restr Y=f_Z\restr (X\cap Y)= f_Y\restr X.$ \end{proof} 

Let $b$ be $\underset{X\in E}{\bigcup} f_X.$ The previous claim implies that $b$ is a function and $$b\restr X= f_X,\textrm{ for all }X\in E.$$ 

\begin{claim} $q\force \dot{b}= b.$
\end{claim}

\begin{proof} We show that for every $X\in E,$ the set $B_X:= \{ s\in \mathbb{P};\ s\force \dot{b}\restr X= b\restr X \}$ is dense below $q.$ Let $r\leq q,$ there is $Y\in E$ such that 
$\dom(r), X\subseteq Y.$ 
It follows that $p_Y\cap r= p_Y\restr Y\cap r= q\cap r= q,$ thus $p_Y\vert \vert r.$ Let $s\leq p_Y, r,$ then $s\in B_X,$ because $s\force \dot{b}\restr X= f_Y\restr X= f_X= b\restr X.$ 
Since $\bigcup\{X\cap \mu; X\in E\}= \mu,$ we have $q\force \dot{b}= b.$
\end{proof}

That completes the proof. \end{proof}

\begin{theorem}\label{secondo} Let $V\subseteq W$ be two models of set theory with the same ordinals and let $\mathbb{P}\in V$ be a forcing notion and $\kappa$ a cardinal in $V$ such that: 
\begin{enumerate}
\item $\mathbb{P}\subseteq \Add(\aleph_n, \tau)^V,$ for some $\tau>\aleph_n,$\\
 and for every $p\in \mathbb{P},$ if $X\subseteq \dom(p),$ then $p\restr X\in \mathbb{P},$
\item $\aleph_{m}^V=\aleph_m^W,$ for every $m\leq n,$ and $W\models \vert \kappa \vert=\aleph_{n+1},$
\item for every set $X\subseteq V$ in $W$ of size $<\aleph_{n+1}$ in $W,$ there is $Y\in V$ of size $<\kappa$ in $V,$ such that $X\subseteq Y,$
\item in $V,$ we have $\gamma^{<\aleph_n}<\kappa,$ for every cardinal $\gamma<\kappa.$ 
\end{enumerate}
Let $F\in W$ be a $(\aleph_{n+1},\mu )$-tree with $\mu\geq \aleph_{n+1},$ then for every filter $G_{\mathbb{P}}\subseteq \mathbb{P}$ generic over $W,$ every cofinal branch for $F$ in $W[G_{\mathbb{P}}]$ is already in $W.$
\end{theorem}

\begin{proof} Same proof as for the Second Preservation Theorem. \end{proof}




\section{The Main Forcing}\label{mainforcing}

\begin{definition} Let $\eta$ be a regular cardinal and let $\theta$ be any ordinal, we define
$$\mathbb{P}(\eta, \theta):=\{p\in \Add(\eta, \theta);\ \textrm{for every $\alpha\in\dom(p),$ $\alpha$ is a successor ordinal }\},$$
$\mathbb{P}(\eta, \theta)$ is ordered by reverse inclusion. \end{definition}

For $E\subseteq \theta,$ we denote by $P(\eta, \theta)\restr E$ the set of all functions in $P(\eta, \theta)$ with domain a subset of $E.$ The following definition is due to Abraham \cite{Abraham}. 

\begin{definition}\label{Mitchell}Assume that $V\subseteq W$ are two models of set theory with the same ordinals, let $\eta$ be a regular cardinal in $W$ and let 
$\mathbb{P}:=\mathbb{P}(\eta, \theta)^{V},$ where $\theta$ is any ordinal. We define in $W$ the poset $\mathbb{M}(\eta, \theta, V, W)$ as follows:\\

$(p,q)\in \mathbb{M}(\eta, \theta, V, W)$ if, and only if,  
\begin{enumerate}
\item $p\in \mathbb{P}(\eta, \theta)^{V};$
\item $q\in W$ is a partial function on $\theta$ of size $\leq \eta$ such that for every $\alpha\in \dom(q),$ $\alpha$ is a successor ordinal, $q(\alpha)\in W^{\mathbb{P}\restr \alpha},$ and  
$\force_{\mathbb{P}\restr \alpha}^{W} q(\alpha) \in \Add(\eta^+)^{V[\redux{\mathbb{P}\restr \alpha}]}.$ 
\end{enumerate}

$\mathbb{M}(\eta, \theta, V, W)$ is partially ordered by $(p, q)\leq (p', q')$ if, and only if, 
\begin{enumerate}
\item $p\leq p';$
\item $\dom(q')\subseteq \dom(q);$
\item $p\restr \alpha\force_{P\restr \alpha}^{W} q(\alpha)\leq q'(\alpha),$ for all $\alpha\in \dom(q').$ 
\end{enumerate}
\end{definition}

\

If $\theta$ is a  weakly compact cardinal, then $\mathbb{M}(\aleph_n, \theta, V, V)$ corresponds to the forcing defined by Mitchell for a model of the tree property at $\aleph_{n+2}$ (see \cite{Mitchell}). Weiss proved that a variation of that forcing with $\theta$ supercompact, produces a model of the super tree property for $\aleph_{n+2}.$ Let's discuss a naive attempt to build a model of the super tree property for two successive cardinals $\aleph_n,$ $\aleph_{n+1}$ (with $n\geq 2$). We start with two supercompact cardinals 
$\kappa<\lambda$ in a model $V,$ then we force with $\mathbb{M}(\aleph_{n-2}, \kappa, V, V)$ over $V$ obtaining a model $W;$ finally, we force over $W$ with 
$\mathbb{M}(\aleph_{n-1}, \lambda, W, W).$ The problem with this approach is that the second stage might introduce an $(\aleph_n, \mu)$-tree $F$ with no cofinal branches. Therefore, we have to define the first stage of the iteration so that it will make the \itp at $\aleph_n$ ``indestructible''. The forcing notion required for that ``anticipates'' at the first stage a fragment of 
$\mathbb{M}(\aleph_{n-1}, \lambda, W, W).$ 

\begin{definition} For $V,W$ and $\eta, \theta$ like in Definition \ref{Mitchell}, we define 
$$\mathbb{Q}^*(\eta, \theta, V, W):= \{(\emptyset, q);\ (\emptyset, q)\in \mathbb{M}(\eta, \theta, V, W) \}.$$
\end{definition}

The poset defined hereafter is a variation of the forcing construction defined by Abraham in \cite[Definition 2.14]{Abraham}. 

\begin{definition}\label{nuovaiterazione} Let $V$ be a model of set theory, and suppose that $\theta>\aleph_n$ is an inaccessible cardinal. 
Let $\mathbb{P}:= \mathbb{P}(\aleph_n, \theta)^V$ and let $L:\theta \to V_\theta$ be any function. Define  
$$\mathbb{R}:= \mathbb{R}(\aleph_n, \theta, L)$$ 
as follows. For each $\beta\leq \theta,$ we define by induction $\mathbb{R}\restr \beta$ and then we set $\mathbb{R}= \mathbb{R}\restr \kappa.$\\
 
$\mathbb{R}\restr 0$ is the trivial forcing.\\
$(p,q,f)\in \mathbb{R}\restr \beta$ if, and only if
\begin{enumerate}
\item $p\in \mathbb{P}\restr \beta(= \mathbb{P}(\aleph_n, \beta)^V);$
\item $q$ is a partial function on $\beta$ of size $\leq \aleph_n,$ such that for every $\alpha\in \dom(q),$ $\alpha$ is a successor ordinal, $q(\alpha)\in V^{\mathbb{P}\restr \alpha}$ and 
$\force_{\mathbb{P}\restr \alpha} q(\alpha) \in \Add(\aleph_{n+1})^{V[\redux{\mathbb{P}\restr \alpha}]};$ 
\item $f$ is a partial function on $\beta$ of size $\leq \aleph_n$ such that for all $\alpha\in \dom(f),$ $\alpha$ is a limit ordinal, $f(\alpha)\in V^{\mathbb{R}\restr \alpha}$ and 
$$\force_{\mathbb{R}\restr \alpha} L(\alpha)\textrm{ is an ordinal such that } f(\alpha)\in \mathbb{Q}^*(\aleph_{n+1}^{V[\redux{\mathbb{R}\restr \alpha}]}, L(\alpha), V, V[\redux{\mathbb{R}\restr \alpha}]).$$ 
\end{enumerate}

\

$\mathbb{R}\restr {\beta}$ is partially ordered by $(p,q,f)\leq (p',q',f')$ if, and only if: 
\begin{enumerate}
\item $p\leq p';$
\item $\dom(q')\subseteq \dom(q);$
\item $p\restr \alpha\force_{P\restr \alpha} q(\alpha)\leq q'(\alpha),$ for all $\alpha\in \dom(q').$ 
\item $\dom(f')\subseteq \dom(f);$ 
\item for all $\alpha\in \dom(f'),$ if $(p,q,f)\restr \alpha:= (p\restr \alpha, q\restr \alpha, f\restr \alpha),$ then 
$$(p,q,f)\restr \alpha\force_{\mathbb{R}\restr {\alpha}} f(\alpha)\leq f'(\alpha)$$ 
\end{enumerate}
\end{definition}


Assume that $V$ is a model of $\ZFC$ with two supercompact cardinals $\kappa<\lambda,$ and $L: \kappa\to V_\kappa$ is the Laver function. Let $\mathbb{R}:= \mathbb{R}(\aleph_0, \kappa, L )$ and let $G_{\mathbb{R}}\subseteq \mathbb{R}$ be any generic filter over $V.$ Assume that $G_{\mathbb{M}}$ is an $\mathbb{M}(\aleph_1, \lambda, V, V[G_{\mathbb{R}}])$-generic filter over $V[G_{\mathbb{R}}],$ we will prove in \S \ref{theorem} that both $\aleph_2$ and $\aleph_3$ satisfy the super tree property in 
$V[G_{\mathbb{R}}][G_{\mathbb{M}}].$

\section{Factoring Mitchell's Forcing}\label{bricks}

In this section, $V, W, \eta, \theta$ are like in Definition \ref{Mitchell}. None of the result of this section are due to the author. For more details see \cite{Abraham}.


\begin{remark}\label{stracotta} The function $\pi:\ \mathbb{M}(\eta, \theta, V,W)\to \mathbb{P}(\eta, \theta)^{V}$ defined by $\pi(p,q):=p$ is a projection. If 
$\mathbb{P}:= \mathbb{P}(\eta, \theta)^{V}$ and if $G_\mathbb{P}$ is a $\mathbb{P}$-generic filter over $W,$ then we define in $W[G_\mathbb{P}]$ the poset 
$$\mathbb{Q}(\eta, \theta, V,W, G_{\mathbb{P}}):= \mathbb{M}(\eta, \theta, V,W)/ G_\mathbb{P}.$$ 
\end{remark}


\begin{lemma}\label{proiezione} The function $\sigma: \mathbb{P}(\eta, \theta)^{V}\times \mathbb{Q}^*(\eta, \theta, V,W)\to \mathbb{M}(\eta, \theta, V,W)$ defined by 
$\sigma(p, (\emptyset, q )):= (p,q)$ is a projection. If $G_\mathbb{M}$ is a $W$-generic filter over $\mathbb{M}(\eta, \theta, V,W),$ then we define in $W[G_\mathbb{M}]$ the poset: 
$$\mathbb{S}(\eta, \theta, V,W, G_{\mathbb{M}}):= (\mathbb{P}(\eta, \theta)^{V}\times \mathbb{Q}^*(\eta, \theta, V,W))/ G_\mathbb{M}.$$ 
\end{lemma}

 \begin{proof} Let $\mathbb{P}:= \mathbb{P}(\eta,\theta)^V$ and $\mathbb{Q}^*:= \mathbb{Q}^*(\eta,\theta, V,W).$ It is clear that $\sigma$ preserves the identity and respect the ordering relation. Let $(p',q')\leq \sigma(p, (\emptyset, q)).$ Define $q^*$ as follows: $\dom(q^*)= \dom(q')$ and for $\alpha\in \dom(q'),$ if $\alpha\notin \dom(q),$ then $q^*(\alpha):= q'(\alpha);$ if 
 $\alpha\in \dom(q),$ we define $q^*(\alpha)\in W^{\mathbb{P}\restr \alpha}$ such that the following hold: 
 \begin{enumerate}
 \item $p'\restr \alpha \force q^*(\alpha)= q'(\alpha),$ 
 \item  if $r\in \mathbb{P}\restr \alpha$ is incompatible with $p'\restr \alpha,$ then $r\force q^*(\alpha)= q(\alpha).$
 \end{enumerate}
  So $\force_{\mathbb{P}\restr \alpha}^{W} q^*(\alpha)\leq q(\alpha),$ hence $(p', (\emptyset, q^*))\leq (p, (\emptyset, q))$ in $\mathbb{P}\times \mathbb{Q}^*$ and 
 $\sigma(p', (\emptyset, q^*))=(p',q^*).$ Moreover $[(p', q^*)]= [(p', q')],$ that completes the proof. \end{proof}
  
\begin{lemma}\label{dirclosed} $\mathbb{Q}^*(\eta, \theta, V,W)$ is $\eta^+$-directed closed in $W.$
\end{lemma}

\begin{proof} See \cite{Abraham} for a proof of that lemma. \end{proof}

\begin{lemma}\label{nuovo} Assume that $\mathbb{P}:= \mathbb{P}(\eta,\theta)^V$ is $\eta^+$-cc in $W,$ for every filter $G_{\mathbb{M}}\subseteq \mathbb{M}(\eta, \theta, V, W)$ generic over $W,$ if $G_{\mathbb{P}}\subseteq \mathbb{P}$ is the projection of $G_{\mathbb{M}}$ to $\mathbb{P},$ then all sets of ordinals in $W[G_{\mathbb{M}}]$ of size $\eta$ are in $W[G_{\mathbb{P}}].$ 
\end{lemma}

\begin{proof} By Lemma \ref{proiezione} it is enough to prove that if $G_{\mathbb{P}}\times G_\mathbb{Q}\subseteq \mathbb{P}\times \mathbb{Q}^*(\eta, \theta, V, W)$ is a generic filter over $W,$
then every set of ordinals in $W[G_{\mathbb{P}}\times G_\mathbb{Q}]$ of size $\eta$ is already in $W[G_{\mathbb{P}}].$ This is an easy consequence of the Easton's Lemma. 
\end{proof}

\begin{proposition}\label{collapse1} Assume that $\theta$ is inaccessible in $W$ and let $\mathbb{M}:= \mathbb{M}(\eta, \theta, V, W).$ The following hold:
\begin{enumerate}
\item[(i)] $\vert \mathbb{M}\vert = \theta$ and $\mathbb{M}$ is $\theta$-c.c.; 
\item[(ii)] If $\mathbb{P}(\eta, \theta)^V$ is $\eta^+$-cc in $W,$ then $\mathbb{M}$ preserves $\eta^+;$
\item[(iii)] If $\mathbb{P}(\eta, \theta)^V$ is $\eta^+$-c.c. in $W,$ then $\mathbb{M}$ makes $\theta= \eta^{++}= 2^{\eta}.$
\end{enumerate}
\end{proposition}

\begin{proof} $(i)$ The proof that $\vert \mathbb{M}\vert= \theta $ is omitted. The key point is that since $\kappa$ is inaccessible, then $\mathbb{P}(\eta, \theta)$ has size 
$\theta$ and for every $(p,q)\in \mathbb{M},$ there are fewer than $\theta$ possibilities for $q(\alpha).$ 
The proof that $\mathbb{M}$ is $\theta$-c.c. is a standard application of the $\Delta$-system Lemma.\\ 
\indent $(ii)$ It follows from Lemma \ref{nuovo}.\\
\indent $(iii)$ For every cardinal $\alpha\in ]\eta, \theta[,$ $\mathbb{M}$ projects to $\mathbb{P}(\eta, \alpha)^V$ which makes $2^{\eta}\geq \alpha$ and then adds a Cohen subset of $\eta^+.$ That forcing will collapse $\alpha$ to $\eta^+.$ By the previous claims, $\eta^+$ is preserved and $\theta$ remains a cardinal after forcing with $\mathbb{M}.$ So, 
$\mathbb{M}$ makes $\theta=\eta^{++}.$ \end{proof}

\begin{lemma}\label{provaa} The following hold:

\begin{enumerate}
\item\label{stracotta1} Assume that $\mathbb{P}:= \mathbb{P}(\eta, \theta)^{V},$ if $\mathbb{P}$ adds no new $< \eta$ sequences to $W,$ then 
$$\force_{\mathbb{P}}^W \mathbb{Q}(\eta, \theta, V,W, \redux{\mathbb{P}})\textrm{ is $\eta$-closed};$$

\item\label{S} Assume that $\mathbb{P}:= \mathbb{P}(\eta, \theta)^{V}$ and $\mathbb{M}:= \mathbb{M}(\eta, \theta, V,W).$ If $\mathbb{P}$ adds no new 
$< \eta$ sequences to $W,$ then $\force_{\mathbb{M}}^W \mathbb{S}(\eta, \theta, V,W, \redux{\mathbb{M}})\textrm{ is $\eta$-closed}.$
\end{enumerate}

\end{lemma}

\begin{proof} See \cite{Abraham}. \end{proof}

\

For any ordinal $\alpha\in ]\eta, \theta[,$ 
the function $(p,q)\mapsto (p\restr \alpha, q\restr \alpha)$ is a projection from $\mathbb{M}(\eta, \theta, V,W)$ to $\mathbb{M}_{\alpha}:= \mathbb{M}(\eta, \alpha, V,W).$ 
We want to analyse 
$$\mathbb{M}(\eta, \theta, V,W)/ G_{\mathbb{M}_{\alpha}},$$
where $G_{\mathbb{M}_{\alpha}}\subseteq \mathbb{M}_{\alpha}$ is any generic filter over $W.$ 
Consider the following definition.

\begin{definition} Let $\theta'\in ]\eta, \theta[$ be any ordinal and let $\mathbb{P}:=\mathbb{P}(\eta, \theta)^{V}.$ Let $\mathbb{M}_{\theta'}:= \mathbb{M}(\eta, \theta', V,W)$ and assume that 
$G_{\mathbb{M}_{\theta'}}\subseteq \mathbb{M}_{\theta'}$ is any generic filter over $W,$ then we define in $W':= W[G_{\mathbb{M}_{\theta'}}],$ the following poset $\mathbb{M}(\eta, \theta-\theta', V, W').$\\
 
$(p,q)\in \mathbb{M}(\eta, \theta-\theta', V, W')$ if, and only if,
\begin{enumerate}
\item $p\in \mathbb{P}\restr (\theta-\theta');$
\item $q\in W'$ is a partial function on $]\theta', \theta[$ of size $\leq \eta$ such that for every\linebreak
$\alpha\in \dom(q),$ $\alpha$ is a successor ordinal, 
$q(\alpha)\in (W')^{\mathbb{P}\restr (\alpha-\theta')},$ and 
$$\force_{\mathbb{P}\restr (\alpha-\theta')}^{W'} q(\alpha)\in \Add(\eta^+)^{W'[\redux{\mathbb{P}\restr (\alpha-\theta')}]}.$$
\end{enumerate}

$\mathbb{M}(\eta, \theta-\theta', V,W)$ is partially ordered as in Definition \ref{Mitchell}.
\end{definition}

\begin{lemma}\cite[Lemma 2.12]{Abraham}\label{Abraham} Let $\theta'\in ]\eta, \theta[$ be any ordinal and let $\mathbb{M}_{\theta'}:= \mathbb{M}(\eta, \theta', V,W)$ with $G_{\mathbb{M}_{\theta'}}\subseteq \mathbb{M}_{\theta'}$ a generic filter over 
$W.$ Assume that $\mathbb{P}(\eta, \theta)$ is $\eta^+$-cc in $W$ and in $W[G_{\mathbb{M}_{\theta'}}],$ then 
$$\mathbb{M}(\eta, \theta, V,W)\equiv \mathbb{M}_{\theta'}\ast \mathbb{M}(\eta, \theta- \theta', V, W[\redux{\mathbb{M}_{\theta'}}]).$$
\end{lemma}

\begin{proof} One can prove that $\mathbb{M}_{\theta'}\ast \mathbb{M}(\eta, \theta-\theta', V, W[\redux{\mathbb{M}_{\theta'}}] )$ contains a dense set isomorphic to 
$\mathbb{M}(\eta,\theta, V, W).$ The proof is omitted, for more details see \cite{Abraham} Lemma 2.12. \end{proof}

\begin{remark}\label{Brionia1} Lemma \ref{proiezione} and Lemma \ref{dirclosed}, can be generalized in the following way. Assume that $\theta'<\theta,$ 
$\mathbb{P}:= \mathbb{P}(\eta, \theta)^{V}\restr (\theta-\theta'),$ $\mathbb{M}_{\theta'}:= \mathbb{M}(\eta, \theta', V,W)$ and $G_{\mathbb{M}_{\theta'}}\subseteq \mathbb{M}_{\theta'}$ is a generic filter over $W,$ 
define $$\mathbb{Q}^*(\eta, \theta-\theta', V, W[G_{\mathbb{M}_{\theta'}}]):=\{(\emptyset, q);\spazio (\emptyset, q)\in \mathbb{M}(\eta, \theta-\theta', V,W[G_{\mathbb{M}_{\theta'}}]) \}.$$ 
Then,  $\mathbb{M}(\eta, \theta-\theta', V, W[G_{\mathbb{M}_{\theta'}}])$ is a projection of $\mathbb{P}\times \mathbb{Q}^*(\eta, \theta-\theta', V,W[G_{\mathbb{M}_{\theta'}}])$ and 
$\mathbb{Q}^*(\eta, \theta-\theta', V,W[G_{\mathbb{M}_{\theta'}}])$ is $\eta^+$-directed closed in $W[G_{\mathbb{M}_{\theta'}}].$
\end{remark}

\section{Factoring the Main Forcing}\label{bricks2}
 
 In this section $\theta, V, L$ are like in Definition \ref{nuovaiterazione}. We want to analyse the forcing $\mathbb{R}(\aleph_0, \theta, L).$ As we said, that poset is a variation of the forcing defined by Abraham in \cite[Definition 2.14]{Abraham}, we have just to deal with the function $L.$ The proofs of the lemmas presented in this section are very similar to the proofs of the corresponding lemmas in \cite{Abraham}.
 
\begin{remark}\label{key} $(p,q,f)\mapsto (p,q)$ is a projection of $\mathbb{R}(\aleph_0, \theta, L)$ to $\mathbb{M}(\aleph_0, \theta, V, V)$ and for every limit ordinal 
$\alpha<\theta,$ if $L(\alpha)$ is an $\mathbb{R}\restr \alpha$-name for an ordinal and $\mathbb{Q}^*$ is the canonical $\mathbb{R}\restr \alpha$-name for 
$\mathbb{Q}^*(\aleph_1, L(\alpha), V, V[\redux{R\restr {\alpha}}]),$ then  
$$\mathbb{R}\restr {\alpha+1}= \mathbb{R}\restr {\alpha}\ast \mathbb{Q}^*.$$
Indeed, the functions in $\mathbb{M}(\aleph_0,\theta, V, V)$ are not defined on limit ordinals.  
\end{remark}

\begin{lemma}\label{nuovo2} Let $\mathbb{U}(\aleph_0, \theta, L):= \{(\emptyset,q,f);\ (\emptyset,q,f)\in \mathbb{R} \}$ ordered as a subposet of $\mathbb{R}.$ The following hold: 
\begin{enumerate}
\item[(i)] the function $\pi: \mathbb{P}(\aleph_0, \theta)\times \mathbb{U}(\aleph_0, \theta, L)\to \mathbb{R}$ defined by $\pi(p, (\emptyset,q,f))= (p,q,f)$ is a projection;
\item[(ii)] $\mathbb{U}(\aleph_0, \theta, L)$ is $\sigma$-closed.
\end{enumerate}
\end{lemma}

\begin{proof} $(i)$ Let $(p', q', f')\leq \pi(p,(\emptyset, q,f)).$ By Lemma \ref{proiezione}, the function 
$(p, (\emptyset, q))\mapsto (p,q)$ is a projection and we can find $(\emptyset, q^*)\leq (\emptyset, q)$ such that
$[(p', q^*)]= [(p',q')].$ We define a function $f^*$ as follows: $\dom(f^*)= \dom(f')$ and for all $\alpha\in \dom(f'),$ if $\alpha\notin \dom(f),$ then $f^*(\alpha):= f'(\alpha);$ if 
 $\alpha\in \dom(f),$ we define $f^*(\alpha)\in V^{\mathbb{R}\restr \alpha}$ such that the following hold: 
 \begin{enumerate}
 \item $(p', q', f')\restr \alpha \force_{\mathbb{R}\restr \alpha} f^*(\alpha)= f'(\alpha),$ 
 \item  if $r\in \mathbb{R}\restr \alpha$ is incompatible with $(p',q',f')\restr \alpha,$ then $r\force_{\mathbb{R}\restr \alpha} f^*(\alpha)= f(\alpha).$
 \end{enumerate}
  Since $(p',q',f')\restr \alpha\force_{\mathbb{R}\restr \alpha} f'(\alpha)\leq f(\alpha),$ then $\force_{\mathbb{R}\restr \alpha} f^*(\alpha)\leq f(\alpha).$
One can prove by induction on $\alpha$ that $[(p^*, q^*, f^*)\restr \alpha]= [(p', q', f')\restr \alpha],$ and we have $(\emptyset, q^*, f^*)\leq (\emptyset, q,f).$\\ 
\indent $(ii)$ Let $\langle (\emptyset, q_n, f_n);\ n<\omega\rangle$ be a decreasing sequence of conditions in $\mathbb{U}(\aleph_0, \theta, L).$ By definition, 
$\langle (\emptyset, q_n);\ n<\omega  \rangle$ is a decreasing sequence of conditions in $\mathbb{Q}^*(\aleph_0, \theta, V, V)$ which is $\sigma$-closed by Lemma \ref{dirclosed}. So there 
is $(\emptyset, q)$ such that $(\emptyset, q)\leq (\emptyset, q_n),$ for every $n<\omega.$ We define a function $f$ with $\dom(f)= \underset{n<\omega}\bigcup \dom(f_n)$ as follows. We define $f\restr \alpha+1$ by induction on $\alpha,$ so that 
$$(\emptyset, q\restr \alpha+1, f\restr \alpha+1)\leq (\emptyset, q_n, f_n)\restr \alpha+1,$$ 
for all $n<\omega.$ Assume $f\restr \alpha$ has been defined. For every $m>n,$ we have 
$$(\emptyset, q_m, f_m)\restr \alpha\force_{\mathbb{R}\restr \alpha} f_m(\alpha)\leq f_n(\alpha),$$
so by the inductive hypothesis we have $(\emptyset, q\restr \alpha, f\restr \alpha)\force f_m(\alpha)\leq f_n(\alpha).$ By Lemma \ref{dirclosed}, if $G_{\alpha}\subseteq \mathbb{R}\restr \alpha$ is a generic filter over $V,$ then $\mathbb{Q}^*(\aleph_1, L(\alpha), V, V[G_{\alpha}])$ is $\aleph_2$-closed in $V[G_{\alpha}].$ It follows that for some $f(\alpha)\in V^{\mathbb{R}\restr \alpha},$ we have 
$$(\emptyset, q\restr \alpha, f\restr \alpha)\force f(\alpha)\leq f_m(\alpha),\textrm{ for every }m<\omega.$$ 
Finally, the condition $(\emptyset, q, f)$ is a lower bound for $\langle (\emptyset, q_n, f_n);\ n<\omega\rangle.$ \end{proof}

\begin{lemma}\label{prova} Assume that $V$ is a model of $\ZFC$ with two supercompact cardinals $\kappa<\lambda,$ and $L: \kappa\to V_\kappa$ is the Laver function. Let
$\mathbb{R}:= \mathbb{R}(\aleph_0, \kappa, L ),$ and let $\dot{\mathbb{M}}$ be the canonical $\mathbb{R}$-name for 
$\mathbb{M}(\aleph_1, \lambda, V, V[\redux{\mathbb{R}}]).$ The following hold:
\begin{enumerate}

\item\label{fede} $\mathbb{R}$ has size $\kappa$ and it is $\kappa$-c.c.; 

\item\label{federica} $\force_{\mathbb{R}} \lambda \textrm{ is inaccessible;}$

\item\label{nuovo3} For every filter $G_{\mathbb{R}}\subseteq \mathbb{R}$ generic over $V,$ if $G_0$ is the projection of $G_{\mathbb{R}}$ to 
$\mathbb{P}_0:= \mathbb{P}(\aleph_0, \kappa),$ then all countable sets of ordinals in $V[G_{\mathbb{R}}]$ are in $V[G_0];$  

\item\label{alefunopreservato} $\mathbb{R}$ preserves $\aleph_1$ and makes $\kappa= \aleph_2=2^{\aleph_0};$

\item\label{saggezza} If $G_{\mathbb{R}}\subseteq \mathbb{R}$ is a generic filter over $V,$ then $\mathbb{P}_1:= \mathbb{P}(\aleph_1, \lambda)^V$ does not introduce new countable subsets to $V[G_\mathbb{R}];$

\item\label{zio} $\force_{\mathbb{R}} \mathbb{P}(\aleph_1, \lambda)^V$ is $\kappa$-c.c. (and even $\kappa$-Knaster). 

 \end{enumerate}
\end{lemma}

\begin{proof} (\ref{fede})  The proof is similar to the proof of Lemma \ref{collapse1} (i) and it is omitted.\\
 \indent (\ref{federica}) It follows from the previous claim.\\
 \indent (\ref{nuovo3}) By Lemma \ref{nuovo2}, it is enough to prove that if $G_0\times H\subseteq \mathbb{P}_0\times \mathbb{U}(\aleph_0, \kappa, L)$ is any generic filter over $V,$ then every countable set of ordinals in $V[G_0\times H]$ is already in $V[G_0].$ This is an easy consequence of Easton's Lemma.\\  
 \indent (\ref{alefunopreservato}) Since $\mathbb{P}(\aleph_0, \kappa)$ is c.c.c., Claim \ref{nuovo3} implies that $\aleph_1$ is preserved. Since $\mathbb{R}$ is $\kappa$-c.c., then 
$\kappa$ remains a cardinal after forcing with $\mathbb{R}.$ Moreover, $\mathbb{R}$ projects on 
$\mathbb{M}(\aleph_0, \kappa, V, V)$ that, by Proposition \ref{collapse1}, collapses all the cardinals between $\aleph_1$ and $\kappa$ and adds $\kappa$ many Cohen reals. Therefore $\mathbb{R}$ makes $\kappa= \aleph_2= 2^{\aleph_0}.$ \\
 \indent (\ref{saggezza}) By Lemma \ref{nuovo2}, $\mathbb{R}$ is a projection of $\mathbb{P}_0\times \mathbb{U}_0,$ where 
$\mathbb{P}_0:= \mathbb{P}(\aleph_0, \kappa)$ and $\mathbb{U}:= \mathbb{U}(\aleph_0, \kappa, L).$ By Easton's Lemma 
$\force_{\mathbb{P}_0\times \mathbb{U}} \mathbb{P}_1 \textrm{ is $<\aleph_1$-distributive},$ so no countable sequence of ordinals is added by $\mathbb{P}_1$ to 
$V[G_0\times H],$ where $G_0\subseteq \mathbb{P}_0$ and $H\subseteq \mathbb{U}$ are generic filters over $V$ such that $G_{\mathbb{R}}$ is the projection of $G_0\times H$ to $\mathbb{R}.$ 
Moreover, we proved in Claim \ref{nuovo3}, 
that every countable sequence of ordinals in $V[G_0\times H]$ is already in $V[G_0].$ Since $V[G_0]\subseteq V[G_{\mathbb{R}}],$ this completes the proof.\\   
\indent (\ref{zio}) Let $G_{\mathbb{R}}\subseteq \mathbb{R}$ be a generic filter over $V.$ Work in $V[G_{\mathbb{R}}].$ 
Assume that $\langle f_{\alpha};\ \alpha<\kappa \rangle$ is a sequence of conditions in $\mathbb{\mathbb{P}}_1:= \mathbb{P}(\aleph_1, \lambda)^V.$ Let 
$D:= \underset{\alpha<\kappa}{\bigcup} \dom(f_{\alpha}),$ then there is a bijection $h: D\to \kappa.$ Since every condition of the sequence is a countable function we have, for every $\alpha<\kappa$ of uncountable cofinality $\sup(h''[\dom(f_{\alpha})]\cap \alpha)<\alpha.$ So the function $\alpha\mapsto \sup(h''[\dom(f_{\alpha})]\cap \alpha)$ is regressive. By Fodor's Theorem, there is an ordinal $\tau$ and a stationary set $S\subseteq \kappa$ such that $\sup(h''[\dom(f_{\alpha})]\cap \alpha)=\tau,$ for every $\alpha\in S.$ 
The set $h^{-1}(\tau)$ has size $<\kappa$ in $V[G_{\mathbb{R}}]$ and $\mathbb{R}$ is $\kappa$-c.c., so there is a set $E\in V$ of size $<\kappa$ in $V$ such that 
$h^{-1}(\tau)\subseteq E.$ Since $\kappa$ is inaccessible in $V,$ then we can find in $V[G_{\mathbb{R}}]$ a stationary set $S'\subseteq S$ such that $f_{\alpha}\restr E$ has a fixed value, for every $\alpha\in S'.$ Then the sets in $\{\dom(f_{\alpha})\setminus E;\ \alpha\in S' \}$ can be assumed to be pairwise disjoint, hence $f_{\alpha}\cup f_{\beta}$ is a function for every $\alpha, \beta\in S'.$ \end{proof}

\begin{lemma}\cite[Lemma 2.18]{Abraham}\label{stanca} Assume that $\alpha<\theta$ is a limit ordinal, let $\mathbb{P}:= \mathbb{P}(\aleph_0, \theta)\restr (\theta- \alpha),$ $\mathbb{R}:=\mathbb{R}(\aleph_0, \theta, L)$ and let $G_{\alpha}\subseteq \mathbb{R}\restr \alpha+1$ be a generic filter over $V.$ We define in $V[G_{\alpha+1}]$ the following set: 
$$\mathbb{U}_{\alpha+1}(\aleph_0, \theta, L, G_{\alpha+1}):= \{(0,q,f)\in \mathbb{R}(\aleph_0, \theta, L);\spazio (0,q,f)\restr \alpha+1\in G_{\alpha+1} \}.$$ 
Then $\mathbb{R}/ G_{\alpha+1}$ is a projection of $\mathbb{P}\times \mathbb{U}_{\alpha+1}(\aleph_0, \theta, L, G_{\alpha+1}),$ and 
$\mathbb{U}_{\alpha+1}(\aleph_0, \theta, L, G_{\alpha+1})$ is $\sigma$-closed in $V[G_{\alpha+1}].$
\end{lemma}

\begin{proof} The proof is very similar to the proof of Lemma 2.18 in \cite{Abraham} and it is omitted. \end{proof}

  

\section{The Main Theorem}\label{theorem}

\begin{theorem} Assume that $V$ is a model of $\ZFC$ with two supercompact cardinals $\kappa<\lambda,$ and suppose that $L: \kappa\to V_\kappa$ is the Laver function. If 
$\mathbb{R}:= \mathbb{R}(\aleph_0, \kappa, L ),$ and $\dot{\mathbb{M}}$ is the canonical $\mathbb{R}$-name for 
$\mathbb{M}(\aleph_1, \lambda, V, V[\redux{\mathbb{R}}]),$ then for every filter $G\subseteq \mathbb{R}\ast \dot{\mathbb{M}}$ generic over $V,$ both $\aleph_2$ and $\aleph_3$ satisfy the super tree property in $V[G].$  
\end{theorem}

The proof that the model obtained is as required consists of three parts:
\begin{enumerate}
\item $V[G]\models \aleph_1^V= \aleph_1,\ \kappa= \aleph_2=2^{\aleph_0}\textrm{ and }\lambda=\aleph_3= 2^{\aleph_1};$
\item $\aleph_3$ satisfies the super tree property.
\item $\aleph_2$ satisfies the super tree property;
\end{enumerate}

\section*{Proof of (1)}

First we show that $\aleph_1$ is preserved.
Let $G_{\mathbb{R}}$ be the projection of $G$ to $\mathbb{R}$ and let $G_{\mathbb{M}}$ be the projection of $G$ to $\mathbb{M}:= \dot{\mathbb{M}}^{G_{\mathbb{R}}}.$ By Lemma \ref{prova}, 
$\aleph_1$ is preserved by $\mathbb{R}.$ Moreover, $\mathbb{P}(\aleph_1, \lambda)^V$ does not introduce 
new countable subsets to $V[G_{\mathbb{R}}]$ (see Lemma \ref{prova} (\ref{saggezza})). So, by Lemma \ref{provaa} (\ref{stracotta1}) $\mathbb{M}$ does not introduce new countable sequences, hence $\aleph_1$ remains a cardinal in $V[G].$
Now, we show that $\kappa$ remains a cardinal in $V[G].$ By Lemma \ref{prova}, we know that $\kappa$ remains a cardinal in $V[G_{\mathbb{R}}]$ and becomes $\aleph_2.$ 
By Lemma \ref{prova} (\ref{zio}), $\mathbb{P}(\aleph_1, \lambda)^V$ is $\kappa$-c.c. in $V[G_{\mathbb{R}}],$ so $\kappa$ remains a cardinal after forcing with 
$\mathbb{P}(\aleph_1, \lambda)^V$ over $V[G_{\mathbb{R}}]$ and it is equal to $\aleph_2.$ By applying Lemma \ref{nuovo}, we get that all sets of ordinals
in $V[G]$ of cardinality $\aleph_1$ are in $V[G_{\mathbb{R}}][G_{\mathbb{P}}],$ where $G_{\mathbb{P}}$ is the projection of $G_{\mathbb{M}}$ to 
$\mathbb{P}(\aleph_1, \lambda)^V.$ Therefore, $\kappa$ remains a cardinal in $V[G].$
Finally, $\lambda$ remains a cardinal because $\mathbb{R}\ast \dot{\mathbb{M}}$ is $\lambda$-c.c., and it becomes $\aleph_3.$

\section*{Proof of 2}  

By (1), we know that $\lambda=\aleph_3$ in $V[G],$ so we want to prove that $\lambda$ has the super tree property in that model. Let $\mu\geq \lambda$ be any ordinal, and assume towards a contradiction that in $V[G]$ there is a $(\lambda, \mu)$-tree $F$ and an $F$-level sequence $D$ with no ineffable branches. 
Fix an elementary embedding $j: V\to N$ with critical point $\lambda$ such that:
\begin{enumerate}
\item if $\sigma:= \vert \mu\vert^{<\lambda},$ then $j(\lambda)>\sigma,$
\item $^{\sigma}N\subseteq N.$
\end{enumerate} 

\begin{claim} We can lift $j$ to an elementary embedding $j^*: V[G]\to N[H],$ with $H\subseteq j(\mathbb{R}\ast \dot{\mathbb{M}})$ generic over $N.$
\end{claim}

\begin{proof} To simplify the notation we will denote all the extensions of $j$ by ``$j$'' also. We let $G_{\mathbb{R}}$ be the projection of $G$ to $\mathbb{R}$ and let $G_{\mathbb{M}}$ be the projection of $G$ to $\mathbb{M}:= \dot{\mathbb{M}}^{G_{\mathbb{R}}}.$

As $\lambda>\kappa$ and $\vert \mathbb{R}\vert = \kappa,$ we have 
$j(\mathbb{R})=\mathbb{R},$ so we can lift $j$ to an elementary embedding $$j: V[G_\mathbb{R}]\to N[G_\mathbb{R}].$$  
Observe that $j(\mathbb{M})\restr \lambda= \mathbb{M}(\aleph_1, \lambda, N, N[G_{\mathbb{R}}])= \mathbb{M}(\aleph_1, \lambda, V, V[G_{\mathbb{R}}])=\mathbb{M}.$ 
Force over $V[G_{\mathbb{R}}]$ to get a $j(\mathbb{M})$-generic filter $H_{j(\mathbb{M})}$ such that $H_{j(\mathbb{M})}\restr \lambda= G_{\mathbb{M}}.$ 
By Lemma \ref{collapse1} and Lemma \ref{prova} (\ref{federica}), $\mathbb{M}$ is $\lambda$-c.c. in $V[G_{\mathbb{R}}],$ so 
$j\restr \mathbb{M}$ is a complete embedding from $\mathbb{M}$ into $j(\mathbb{M}),$ hence we can lift $j$ to an elementary embedding 
$$j: V[G_{\mathbb{R}}][G_{\mathbb{M}}]\to N[G_{\mathbb{R}}][H_{j(\mathbb{M})}].$$ \end{proof} 

Rename $j^*$ by $j.$ We define $\mathscr{N}_1:=N[G]$ and $\mathscr{N}_2:=N[G_{\mathbb{R}}][H_{j(\mathbb{M})}].$ 
In $\mathscr{N}_2,$ $j(F)$ is a $(j(\lambda), j(\mu))$-tree and $j(D)$ is a $j(F)$-level sequence. By the closure of $N,$  
the tree $F$ and the $F$-level sequence $D$ are in $\mathscr{N}_1,$ and there is no ineffable branch for $D$ in $\mathscr{N}_1.$

\begin{claim} In $\mathscr{N}_2,$ there is an ineffable branch $b$ for $D.$
\end{claim}

\begin{proof} Let $a:= j''[\mu],$ clearly $a\in [j(\mu)]^{<j(\lambda)}.$ Consider $f:= j(D)(a),$ let $b:\mu \to 2$ be the function defined by $b(\alpha):= f(j(\alpha)),$ we show that 
$b$ is an ineffable branch for $D.$  
Assume for a contradiction that in $\mathscr{N}_2$ there is a club 
$C\subseteq [\mu]^{<\vert \lambda\vert }\cap \mathscr{N}_1$ such that $b\restr X\neq D(X),$ for all $X\in C.$ Then by elementarity, $$f\restr X\neq j(D)(X),$$ for all 
$X\in j(C).$ But $a\in j(C)$ and $f= j(D)(a),$ so we have a contradiction. 
\end{proof}

Since there is no ineffable branch for $D$ in $\mathscr{N}_1,$ we get a contradiction by proving that forcing with $\mathbb{M}(\aleph_1, j(\lambda)- \lambda, N, N[G])$ over 
$N[G]$ does not add ineffable branches to $D.$ By Remark \ref{Brionia1}, $\mathbb{M}(\aleph_1, j(\lambda)- \lambda, N, N[G])$ is a projection of 
$$\mathbb{P}\times \mathbb{Q}^*(\aleph_1, j(\lambda)- \lambda, N, N[G]),$$
where $\mathbb{P}:= \mathbb{P}(\aleph_1, j(\lambda))^N\restr (j(\lambda)- \lambda),$ and 
$\mathbb{Q}^*:= \mathbb{Q}^*(\aleph_1, j(\lambda)- \lambda, N, N[G])$ is $\aleph_2$-closed in $N[G].$ 
In $N[G],$ we have $\lambda = \aleph_3= 2^{\aleph_1}$ and $F$ is an $(\aleph_3,\mu)$-tree, so we can apply Corollary \ref{primo}, thus $\mathbb{Q}^*$ cannot add ineffable branches to $D.$\\
 
This means that $b\notin N[G][H_{\mathbb{Q}^*}],$ where $H_{\mathbb{Q}^*}$ is the projection of $H_{j(\mathbb{M})}$ to $\mathbb{Q}^*.$ The filter $H_{\mathbb{Q}^*}$ collapses $\lambda$ (which is $\aleph_3^{N[G]}$)to have size $\aleph_2,$ so now $F$ is an $(\aleph_2, \mu)$-tree. Every set $X\subseteq N$ in $N[G][H_{\mathbb{Q}^*}]$ which has size $<\aleph_2$ in $N[G][H_{\mathbb{Q}^*}]$ is covered by a set $Y\in N$ which has size $<\lambda$ in $N.$   
The model $N[G][H_{j(\mathbb{M})}]$ is the result of forcing with $\mathbb{P}$ over $N[G][H_{\mathbb{Q}^*}]$ and $b$ is in $N[G][H_{j(\mathbb{M})}],$ so by Theorem \ref{secondo}, 
$b\in N[G][H_{\mathbb{Q}^*}],$ a contradiction.\\


This completes the proof of (2).

\section*{Proof of 3} 

By (1), we know that $\kappa=\aleph_2$ in $V[G],$ so we want to prove that $\kappa$ has the super tree property in that model. Let $\mu\geq \kappa$ be any ordinal, and assume towards a contradiction that in $V[G]$ there is a $(\kappa, \mu)$-tree $F$ and an $F$-level sequence $D$ with no ineffable branches. Since $L$ is the Laver function, there is an elementary embedding $j: V\to N$ with critical point $\kappa$ such that:
\begin{enumerate}
\item if $\sigma:= \max(\lambda, \vert \mu\vert^{<\kappa}),$ then $j(\kappa)>\sigma,$
\item $^{\sigma}N\subseteq N,$
\item $j(L)(\kappa)= \lambda.$
\end{enumerate} 

\begin{claim}\label{firstclaim} We can lift $j$ to an elementary embedding $j^*: V[G]\to N[H],$ with $H\subseteq j(\mathbb{R}\ast \dot{\mathbb{M}})$ generic over $N.$
\end{claim}

\begin{proof} To simplify the notation we will denote all the extensions of $j$ by ``$j$'' also. Let $G_{\mathbb{R}}$ be the projection of $G$ to $\mathbb{R}$ and let $G_{\mathbb{M}}$ be the projection of $G$ to $\mathbb{M}:= \dot{\mathbb{M}}^{G_{\mathbb{R}}}.$
Observe that  
$j(\mathbb{R})= \mathbb{R}(\aleph_0, j(\kappa), j(L))^N= \mathbb{R}(\aleph_0, j(\kappa), j(L))^V,$ and $j(\mathbb{R})\restr \kappa= \mathbb{R}.$ 
Force over $V$ to get a $j(\mathbb{R})$-generic filter $H_{j(R)}$ such that $H_{j(R)}\restr \kappa= G_{\mathbb{R}}.$ 
By Lemma \ref{prova} (\ref{fede}) $\mathbb{R}$ is $\kappa$-c.c. So $j\restr \mathbb{R}$ is a complete embedding from $\mathbb{R}$ into $j(\mathbb{R}),$ 
hence we can lift $j$ to get an elementary embedding 
$$j: V[G_{\mathbb{R}}]\to N[H_{j(R)}].$$  

By Lemma \ref{proiezione}, in $V[G_{\mathbb{R}}],$ the forcing 
$\mathbb{M}$ is a projection of 
$$\mathbb{P}(\aleph_1, \lambda)^V\times \mathbb{Q}^*(\aleph_1, \lambda, V, V[G_{\mathbb{R}}])$$ 
(moreover, $\mathbb{P}(\aleph_1, \lambda)^V= \mathbb{P}(\aleph_1, \lambda)^N$ and $\mathbb{Q}^*(\aleph_1, \lambda, V, V[G_{\mathbb{R}}])= \mathbb{Q}^*(\aleph_1, \lambda, N, N[G_{\mathbb{R}}])).$ Recall that $$\mathbb{S}(\aleph_1, \lambda, V, V[G_{\mathbb{R}}], G_{\mathbb{M}})= (\mathbb{P}(\aleph_1, \lambda)^V\times \mathbb{Q}^*(\aleph_1, \lambda, V, V[G_{\mathbb{R}}]))/ G_\mathbb{M},$$ so by forcing with $\mathbb{S}(\aleph_1, \lambda, V, V[G_{\mathbb{R}}], G_{\mathbb{M}})$  over $V[G]$ we obtain a model 
$V[G_{\mathbb{R}}][G_{\mathbb{P}}\times G_{\mathbb{Q}^*}]$ with $G_{\mathbb{P}}\times G_{\mathbb{Q}^*}$ generic for 
$\mathbb{P}(\aleph_1, \lambda)^V\times \mathbb{Q}^*(\aleph_1, \lambda, V, V[G_{\mathbb{R}}])$ over $V[G_{\mathbb{R}}]$ and such that $G_{\mathbb{M}}$ is the projection of 
$G_{\mathbb{P}}\times G_{\mathbb{Q}^*}$ to $\mathbb{M}.$\\

If $\mathbb{P}:= \mathbb{P}(\aleph_1, \lambda)^V,$ then $\mathbb{P}$ is $\kappa$-c.c. in $V[G_{\mathbb{R}}]$ (Lemma \ref{prova} (\ref{zio})), hence $j\restr \mathbb{P}$ is a complete embedding of $\mathbb{P}$ into $j(\mathbb{P}).$ Moreover, $\mathbb{P}$ is isomorphic via $j\restr \mathbb{P}$ to 
$\mathbb{P}(\aleph_1, j''[\lambda])^N= \mathbb{P}(\aleph_1, j''[\lambda])^V.$ By forcing with $\mathbb{P}(\aleph_1, j(\lambda))^V\restr (j(\lambda)- j''[\lambda])$ over 
$V[H_{j(\mathbb{R})}]$ we get
a $j(\mathbb{P})$-generic filter $H_{j(\mathbb{P})}$ such that $j''[G_{\mathbb{P}}]\subseteq H_{j(\mathbb{P})}.$ Then $j$ lifts to an elementary embedding   
$$j: V[G_\mathbb{R}][G_{\mathbb{P}}]\to N[H_{j(\mathbb{R})}][H_{j(\mathbb{P})}].$$ 

Let $\mathbb{Q}^*:= \mathbb{Q}^*(\aleph_1, \lambda, V, V[G_{\mathbb{R}}]).$ By Remark \ref{key} and since 
$j(\mathbb{R})\restr \kappa= \mathbb{R},$ we have  
$j(\mathbb{R})\restr \kappa+1= \mathbb{R}\ast \dot{\mathbb{Q}}^*$ where $\dot{\mathbb{Q}}^*$ is an $\mathbb{R}$-name for 
$\mathbb{Q}^*(\aleph_1, j(L)(\kappa), V, V[G_{\mathbb{R}}]).$  
We chose $j$ so that $j(L)(\kappa)= \lambda,$ therefore 
forcing with $j(\mathbb{R})\restr \kappa+1$ over $V$ is the same as forcing with $\mathbb{R}$ followed by forcing with $\mathbb{Q}^*$ over $V[G_{\mathbb{R}}].$
It follows that, by the closure of $N,$ we have $j''[G_{\mathbb{Q}^*}]\in N[H_{j(\mathbb{R})}].$ By Lemma \ref{dirclosed}, $\mathbb{Q}^*$ is $\aleph_2$-directed closed in $V[G_\mathbb{R}],$ hence $j(\mathbb{Q}^*)$ is $\aleph_2$-directed closed in $N[H_{j(\mathbb{R})}].$ Moreover, the filter $H_{j(\mathbb{R})}$ collapses $\lambda$ to have size $\aleph_1,$ thus $j''[G_{\mathbb{Q}^*}]$ has size $\aleph_1$ 
in $V[H_{j(\mathbb{R})}].$ Therefore, we can find 
$t\leq j(q),$ for all $q\in G_{\mathbb{Q}^*}.$ We force over $V[G_{j(\mathbb{R})}]$ with $j(\mathbb{Q}^*)$ below $t$ to get a $j(\mathbb{Q}^*)$-generic filter 
$H_{j(\mathbb{Q}^*)}$ containing $j''[G_{\mathbb{Q}^*}].$ The filter $H_{j(\mathbb{P})}\times H_{j(\mathbb{Q}^*)}$ generates a filter $H_{j(\mathbb{M})}$ generic for 
 $j(\mathbb{M})$ over $N[H_{j(\mathbb{R})}].$\\

It remains to prove that $j''[G_{\mathbb{M}}]\subseteq H_{j(\mathbb{M})}:$ let $(p,q)$ be a condition of $G_{\mathbb{M}},$ there are 
$\bar{p}\in G_{\mathbb{P}}$ and $(0,\bar{q})\in G_{\mathbb{Q}^*}$ such that $(\bar{p}, \bar{q})\leq (p,q).$ We have $j(\bar{p})\in H_{j(\mathbb{P})}$ and 
$(0, j(\bar{q}))\in H_{j(\mathbb{Q}^*)},$ hence $(j(\bar{p}), 0)$ and $(0, j(\bar{q}))$ are both in $H_{j(\mathbb{M})}.$ The condition  
$j(\bar{p}, \bar{q})$ is the greatest lower bound\footnote{
$j(\bar{p}, \bar{q})= (j(\bar{p}), j(\bar{q}))$ is clearly a lower bound. Suppose that $(p_1, q_1)$ is also a lower bound, then by definition 
$p_1\leq j(\bar{p})$ and $p_1\restr \alpha\force q_1(\alpha)\leq j(\bar{q})(\alpha),$ for every $\alpha.$ That is $(p_1, q_1)\leq (j(\bar{p}), j(\bar{q})).$ }
of $(j(\bar{p}),0)$ and $(0, j(\bar{q}))$; it follows that $j(\bar{p}, \bar{q})\in H_{j(\mathbb{M})}.$ We also have $j(\bar{p}, \bar{q})\leq j(p,q),$ hence $j(p,q)\in H_{j(\mathbb{M})}$ as required. Therefore, $j$ lifts to an elementary embedding $$j: V[G_{\mathbb{R}}][G_{\mathbb{M}}]\to N[H_{j(\mathbb{R})}][H_{j(\mathbb{M})}].$$ \end{proof} 
 
Rename $j^*$ by $j.$ We define $\mathscr{N}_1:=N[G]$ and $\mathscr{N}_2:=N[H_{j(\mathbb{R})}][H_{j(\mathbb{M})}].$ 
In $\mathscr{N}_2,$ $j(F)$ is a $(j(\kappa), j(\mu))$-tree and $j(D)$ is a $j(F)$-level sequence. By the closure of $N,$ the tree $F$ and the $F$-level sequence $D$ are in $\mathscr{N}_1,$ and there is no ineffable branch for $D$ in $\mathscr{N}_1.$


\begin{claim} In $\mathscr{N}_2,$ there is an ineffable branch $b$ for $D.$
\end{claim}

\begin{proof} Let $a:= j''[\mu],$ clearly $a\in [j(\mu)]^{<j(\kappa)}.$ Consider $f:= j(D)(a),$ let $b:\mu\to 2$ be the function defined by $b(\alpha):= f(j(\alpha)),$ we show that $b$ is an ineffable branch for $D.$ 
Assume for a contradiction that for some club 
$C\subseteq [\mu]^{<\vert \kappa\vert }\cap \mathscr{N}_1$ we have $b\restr X\neq D(X),$ for all $X\in C.$ Then by elementarity, $$f\restr X\neq j(D)(X),$$ for all $X\in j(C).$ 
But $a\in j(C)$ and $f= j(D)(a),$ so we have a contradiction. \end{proof}


Since there is no ineffable branch for $D$ in $\mathscr{N}_1,$ we get a contradiction with the following claim. 

\begin{claim} $b\in \mathscr{N}_1.$
\end{claim}

\begin{proof} Assume towards a contradiction that $b\notin \mathscr{N}_1.$ By Lemma \ref{prova} (\ref{saggezza}) and Lemma \ref{provaa} (\ref{S}), the poset 
$\mathbb{S}:=\mathbb{S}(\aleph_1, \lambda, N, N[G_{\mathbb{R}}], G_{\mathbb{M}})$ is 
$\sigma$-closed in $\mathscr{N}_1.$ In $\mathscr{N}_1,$ we have $\kappa= \aleph_2= 2^{\aleph_0},$ hence $F$ is a $(\aleph_2,\mu)$-tree and we can apply the First Preservation Theorem to $\mathbb{S},$ thus 
$$b\notin N[G_{\mathbb{R}}][G_{\mathbb{P}}\times G_{\mathbb{Q}^*}]$$
(we defined $G_{\mathbb{P}}\times G_{\mathbb{Q}^*}$ in Claim \ref{firstclaim} as a generic filter for $\mathbb{S}$). $\mathbb{S}$ is $<\aleph_2$-distributive in $\mathscr{N}_1$ (this is a standard application of the Easton's Lemma, see \cite[Lemma 3.20]{CummingsForeman} for more details) so $F$ is still an $(\aleph_2, \mu)$-tree after forcing with $\mathbb{S}.$ Now, the forcing that takes us from $\mathbb{P}$ to $j(\mathbb{P})$ is 
$$\mathbb{P}_{tail}:= \mathbb{P}(\aleph_1, j(\lambda))^N\restr (j(\lambda)- \lambda)).$$ 
The pair $(N, N[G_{\mathbb{R}}][G_{\mathbb{P}}\times G_{\mathbb{Q}^*}])$ has the $\kappa$-covering property, because $\mathbb{S}$ is $<\aleph_2$-distributive and $\mathbb{R}$ is $\kappa$-c.c. Since $\kappa$ is inaccessible in $N,$ we can apply the Second Preservation Theorem to $\mathbb{P}_{tail},$ so $$b\notin N[G_{\mathbb{R}}][G_{\mathbb{Q}^*}][H_{j(\mathbb{P})}].$$


We already observed in the proof of the first claim that forcing with $j(\mathbb{R})\restr \kappa+1$ over $V$ is the same as forcing with 
$\mathbb{R}$ followed by forcing with $\mathbb{Q}^*$ over $V[G_{\mathbb{R}}].$ So, if $H_{\kappa+1}$ is the projection of $H_{j(\mathbb{R})}$ to $j(\mathbb{R})\restr \kappa+1,$ then $N[G_{\mathbb{R}}][G_{\mathbb{Q}^*}]= N[H_{\kappa+1}].$ This means that we proved $$b\notin N[H_{\kappa+1}][H_{j(\mathbb{P})}].$$

Consider $\mathbb{R}_{tail}:= j(\mathbb{R})/ H_{\kappa+1},$ by Lemma \ref{stanca}, $\mathbb{R}_{tail}$ is a projection of $\mathbb{P}_0\times \mathbb{U}_{0},$ 
where $\mathbb{P}_0:= \mathbb{P}(\aleph_0, j(\kappa))^N\restr (j(\kappa)- \kappa)$ and 
$\mathbb{U}_{0}:= \mathbb{U}_{\kappa+1}(\aleph_0, j(\kappa), j(L), H_{\kappa+1} ),$ moreover, $\mathbb{U}_0$ is $\sigma$-closed in $N[H_{\kappa+1}].$ 
Forcing with $j(\mathbb{P})$ does not add countable sequences to $N[H_{\kappa+1}]$ (the proof is analogous to the proof of Lemma \ref{prova} (\ref{saggezza}))
hence $\mathbb{U}_{0}$ is still $\sigma$-closed in $N[H_{\kappa+1}][H_{ j(\mathbb{P})}].$\\

We want to apply the First Preservation Theorem to $\mathbb{U}_0,$ so consider the following facts. In $N[H_{\kappa+1}][H_{j(\mathbb{P})}],$ we have $2^{\aleph_0}= j(\kappa)>\kappa=\aleph_2$ but now $F$ is not exactly an $(\aleph_2, \mu)$-tree because $N[H_{\kappa+1}][H_{j(\mathbb{P})}]$ was obtained by forcing with $\mathbb{P}_{tail}$ over 
$N[G_{\mathbb{R}}][G_{\mathbb{P}}\times G_{\mathbb{Q}^*}]$ and $\mathbb{P}_{tail}$ is not $<\aleph_2$-distributive. Yet, $\mathbb{P}_{tail}$ is $\aleph_2$-c.c. in $N[G_{\mathbb{R}}][G_{\mathbb{P}}\times G_{\mathbb{Q}^*}]$ (the proof is analogous to the proof of Lemma \ref{prova} (\ref{zio}), see \cite{CummingsForeman} for more details), so $F$ "covers" an $(\aleph_2,\mu)$-tree, namely
there is in $N[H_{\kappa+1}][H_{j(\mathbb{P})}]$ a $(\aleph_2, \mu)$-tree $F^*$ such that for cofinally many $X\in [\mu]^{<\aleph_2},$ $\lev_X(F)\subseteq \lev_X(F^*).$  
Let $N[H_{\mathbb{U}_0}][H_{j(\mathbb{P})}]$ be the generic extension obtained by forcing with $\mathbb{U}_0$ over $N[H_{\kappa+1}][H_{j(\mathbb{P})}].$ 
If $b\in N[H_{\mathbb{U}_0}][H_{j(\mathbb{P})}],$ then $b$
provides a cofinal branch for $F^*$ in that model, hence by the First Preservation Theorem, $b\in N[H_{\kappa+1}][H_{j(\mathbb{P})}].$ But we already proved that 
$b$ does not belong to that model, so we must have
$$b\notin N[H_{\mathbb{U}_0}][H_{j(\mathbb{P})}].$$


The filter $H_{\mathbb{U}_0}$ collapses $\kappa$ (hence $\aleph_2$) to have size $\aleph_1,$ so now $F^*$ is an $(\aleph_1, \mu)$-tree in $W:= N[H_{\mathbb{U}_0}][H_{j(\mathbb{P})}].$ 
The model $N[H_{j(\mathbb{R})}][H_{j(\mathbb{P})}]$ is the result of forcing with $\mathbb{P}_0$ over $W.$ Observe that $\mathbb{P}_0$ and $(W, W)$ satisfy all the hypothesis of the Second Preservation Theorem: indeed,  
$\mathbb{P}_0\subseteq \Add(\aleph_0, j(\kappa))^W$ and in $W,$ we have $\gamma^{<\omega}<\omega_1$ for every cardinal $\gamma<\omega_1.$ 
Therefore, $$b\notin N[H_{j(\mathbb{R})}][H_{j(\mathbb{P})}].$$

$\mathbb{P}_0$ is c.c.c. in $W$ (the proof is analogous to the proof of Lemma \ref{prova} (\ref{zio})), so $F^*$ covers an $(\aleph_1,\mu)$-tree in 
$N[H_{j(\mathbb{R})}][H_{j(\mathbb{P})}],$ we rename it $F^*.$
$\mathscr{N}_2$ is the result of forcing with 
$$\mathbb{Q}(\aleph_1, j(\lambda), N, N[H_{j(\mathbb{R})}], H_{j(\mathbb{P})})= \mathbb{M}(\aleph_1, j(\lambda), N, N[H_{j(\mathbb{R})}])/ H_{j(\mathbb{P})}$$ over $N[H_{j(\mathbb{R})}][H_{j(\mathbb{P})}]$ and by Lemma \ref{provaa} (\ref{stracotta1}), that poset is $\sigma$-closed in $N[H_{j(\mathbb{R})}][H_{j(\mathbb{P})}].$ The function $b,$ which is in $\mathscr{N}_2,$ provides a cofinal branch for $F^*$ in $\mathscr{N}_2.$ It follows from the First Preservation Theorem that $b\in N[H_{j(\mathbb{R})}][H_{j(\mathbb{P})}],$ but we proved that $b$ does not belong to that model, so we have a contradiction. 
 \end{proof}   

This completes the proof of (3).

\section{Conclusion}

Cummings and Foreman \cite{CummingsForeman} defined a model of the tree property for every $\aleph_n$ ($n\geq 2$), starting with an infinite sequence of supercompact cardinals 
$\langle \kappa_n \rangle_{n<\omega}.$ Their forcing $\mathbb{R}_{\omega}$ is basically an iteration with length $\omega$ of our main forcing. We conjecture that $\mathbb{R}_{\omega}$ produces a model in which every $\aleph_n$ ($n\geq 2$) satisfies even the super tree property. Yet, if we want to prove that stronger result, we have to deal with the following fact: every $\kappa_n$-tree in the Cummings-Foreman model appears in some intermedium stage, that is after forcing with $\mathbb{R}_{\omega}\restr m$ for some 
$m;$ in the case of a $(\kappa_n, \mu)$-tree, that is not necessarily true.

\section{Acknowledgements}

I would like to thank my advisor Boban Veli\v{c}kovi\'{c} for suggesting to investigate the strong and the super tree property at small cardinals. I owe to him the idea to reformulate the definition of these properties in terms of $(\kappa, \lambda)$-trees, a very useful indication that helped me to prove the results presented in this paper. He gave me many other suggestions for which I am very grateful. 

\bigskip

\end{document}